\documentclass[12pt,a4paper]{amsart}
\usepackage{graphics}
\usepackage{epsfig}
\usepackage{graphicx}
\theoremstyle{plain}
\usepackage{amssymb}
\usepackage{color}
\advance\hoffset-20mm \advance\textwidth40mm

\newtheorem{lemma}{Lemma}
\newtheorem*{theo*}{Theorem}

\theoremstyle{definition}

\newtheorem*{definition*}{Definition}

\begin{document}
\sloppy
\title[Polynomial similarity of pairs of matrices]
{Polynomial similarity of pairs of matrices}
\author
{Vitaliy Bondarenko, Anatoliy  Petravchuk, Maryna Styopochkina}
\address{Institute of Mathematics, National Academy of Sciences of Ukraine,
Tereschenkivska street, 3, 01004 Kyiv, Ukraine}
\email{vitalij.bond@gmail.com}
\address{ Faculty of Mechanics and Mathematics,
Taras Shevchenko National University of Kyiv, 64, Volodymyrska street, 01033  Kyiv, Ukraine}
\email{ petravchuk@knu.ua, apetrav@gmail.com}

\address{Polissia National University,
	Staryi Boulevard, 7, 10008 Zhytomyr, Ukraine}
\email{stmar@ukr.net}

\date{\today}

\begin{abstract}
Let $K$ be a field, $R=K[x, y]$ the polynomial ring and
$\mathcal{M}(K)$ the set of all pairs of square matrices of the same size over $K.$ Pairs $P_1=(A_1,B_1)$ and $P_2=(A_2,B_2)$ from $\mathcal{M}(K)$ are called
similar if $A_2=X^{-1}A_1X$ and
$B_2=X^{-1}B_1X$ for some invertible matrix $X$ over $K$. Denote by
$\mathcal{N}(K)$   the subset  of $\mathcal{M}(K)$, consisting of  all pairs of commuting nilpotent matrices. A pair $P$ will be called  {\it polynomially equivalent} to a pair
$\overline{P}=(\overline{A},  \overline{B})$  if $\overline{A}=f(A,B), \overline{B}=g(A ,B)$ for some polynomials $f, g\in K[x,y]$ satisfying the  next conditions: $f(0,0)=0, g(0,0)=0$ and $ {\rm det} J(f, g)(0, 0)\not =0,$ where $J(f, g)$ is the Jacobi matrix of polynomials $f(x, y)$ and $g(x, y).$ Further, pairs of matrices $P(A,B)$ and $\widetilde{P}(\widetilde{A}, \widetilde{B})$ from $\mathcal{N}(K)$ will be called {\it polynomially similar} if there exists a pair $\overline{P}(\overline{A}, \overline{B})$ from $\mathcal{N}(K)$ such that $P$, $\overline{P}$ are polynomially equivalent and $\overline{P}$,  $\widetilde{P}$ are similar. The main result of the paper: it is   proved that the problem of classifying pairs of matrices up to polynomial similarity is wild, i.e. it contains the classical unsolvable  problem of classifying pairs of matrices up to similarity.

 \end{abstract}
\maketitle


\section{Introduction}

\large

This paper is devoted to  study the  problem  of classifying  pairs of matrices over a field up to a generalized similarity.

Recall that a classification problem is called {\it{wild}} if it contains the problem of classifying pairs of $(n\times n)$-matrices up to simultaneous similarity
$$ (A, B)\longmapsto S^{-1}(A, B)S=(S^{-1}AS, \  S^{-1}BS)  $$
with an invertible matrix $S$.
Otherwise, when  indecomposble objects are  ``parameterized'' by several discrete and at most 
 one continuous parameters, the problem is called {\it tame}.
These concepts (including the terms themselves) were first  introduced by P. Donovan and M. R.  Freislich in  \cite{DonFreis72}.
 Formal defi\-nitions of these two classes were proposed by Yu.A.Drozd
 in  \cite{Drozd77, Drozd79}; in these papers he also proved his well-known theorems about tame and wild problems.
The first results  associated with wild problems  were obtained in
\cite{Krugljak63,
GelfPon69,
Brenner70,
Drozd72,
DonFreis73,
 Nazarova73,
Bondarenko76,
BondarenkoDrozd77}, etc.

Let $R$ be an associative (not necessarily commutative) ring with identity, and let $\mathcal{M}(R)$ be the set of all pairs of square matrices over $R$ of the same size.  If  a pair $(A, B)$ consists of block matrices it is always assumed
that  the partitions of both matrices $A$ and $B$
into  blocks are the same and the diagonal blocks are some  square matrices. An identity block is denoted by  $I$ regardless of its size, but in each specific case it will be easy to find out   whether or not different identity blocks are of the same size.
Let $\mathcal{N}(R)$ denote  the subset  of $\mathcal{M}(R)$ consisting of  all pairs of commuting nilpotent matrices.
We  study the  problem  of classifying  pairs of matrices
from $\mathcal{N}(R)$ in case when $R=K$ is a field, up to similarity of a special kind.

So, let $K$ be a field and
$P=(A, B)$, $\overline{P}=(\overline{A}, \overline{B})$ be two pairs
of matrices from the set $\mathcal{N}(K)$.
We  say that the pair $P$ is  {\it polynomially equivalent} to the pair
$\overline{P}$ and write $P\sim_{K[x,y]} \overline{P}$ if $\overline{A}=f(A,B), \overline{B}=g(A ,B)$
for some polynomials $f, g\in K[x,y]$ satisfying the  next conditions:
\begin{equation}\label{for 00}
f(0,0)=0, g(0,0)=0,
\end{equation}

\begin{equation}\label{for not equal 0}
\left|
\begin{array}{cc}
\frac{df}{dx}(0,0)&\frac{df}{dy}(0,0)\\[2mm]
\frac{dg}{dx}(0,0)&\frac{dg}{dy}(0,0)\\
\end{array}
\right|\ne 0.
\end{equation}

Pairs of matrices $P=(A, B)$ and $\widetilde{P}=(\widetilde{A}, \widetilde{B})$ from
$\mathcal{N}(K)$ will be called {\it polynomially similar},
if there exists a pair
$\overline{P}=(\overline{A}, \overline{B})\in \mathcal{N}(K)$
such that the pairs $P$ and $\overline{P}$ are polynomially equivalent,
and the pairs $\overline{P}$ and  $\widetilde{P}$ are similar.
It can be easily checked that the relation $\sim_{K[x,y]}$ is an equivalence relation on the set $\mathcal{N}(K)$. Indeed, let
$(A,B)\sim_{K[x,y]}  (\overline{A}, \overline{B}).$ Then the radical of the local subalgebra generated by the matrices $A,B$ and by the identity  can be generated by the matrices $A,B$  as well as by  the matrices $\overline{A}, \overline{B}$. The latter means that the relation $\sim_{K[x,y]}$ is symmetric.  The reflexivity and the transitivity of this binary relation are obvious. Since the usual similarity is also an equivalence relation on  $\mathcal{N}(K)$ we see that  the polynomial similarity is an equivalence relation on the set $\mathcal{N}(K)$.

The aim of this paper is to prove that the problem of classifying  pairs of matrices from $\mathcal{N}(K)$
 up to polynomial similarity is wild. Note that this result is
 a far-reaching generalization of the main theorem of
\cite{FKPS} which provides in other terms wildness of our problem
with linear polynomials $f(x,y)$ nnd $g(x,y)$.

\section{Formulation of  the main theorem}

 Let us explain the notion of wildness  more detailed (we use here the matrix language).  Let $F_2=K\langle X, Y\rangle $ be the free associative  non-commutative  algebra in two free generators $ X, Y$ over $K.$  The problem of classifying pairs from $\mathcal{N}(K)$ up to polynomial similarity  is wild if  there exists   a pair $P_0=P_0(X,Y)\in \mathcal{N}(F_2)$ of  matrices $A_0=A_0(X,Y), B_0=B_0(X,Y)$   such that the polynomial similarity of pairs of matrices $P_0(M_1, N_1)$ and $P_0(M_2, N_2)$ with
$$(M_1, N_1), (M_2, N_2)\in \mathcal{M} (K)$$
  implies the similarity of pairs $(M_1, N_1)$ and $(M_2, N_2).$
This definition (in the specific case and in the matrix language) is consistent with the general idea of wildness of classification problems \cite{Drozd77}.
We call such a pair $P_0=P_0(X,Y)\in \mathcal{N}(F_2)$  of matrices  {\it w-defining
relatively the polynomial similarity} or simply {\it w-defining}.

We restrict ourselves to quadratic  polynomials  $f(x, y), g(x, y)$ in the definition of polynomial similarity and
 consider  the problem of classifying  pairs $(A, B)$  of matrices up to polynomial similarity in the case when
$$A^2=0, B^3=0, AB^2=0. \eqno{(*)}$$
{\bf Main theorem.}\label{T1}
{\it The problem of classifying,
 up to polynomial similarity,   pairs of matrices satisfying the equations $(*)$ is wild.}

Note that  in the nearest case $A^2=0, B^2=0$ the problem is not wild  even without the condition of commutativity and under the usual similarity (see the paper \cite{Bondarenko75}).

\section{Proof of the main theorem.}

The  subset of  pairs of matrices $(A, B)$   from
$\mathcal{N}(K)$, satisfying  $(*)$, will be denoted by $\mathcal{N}_{23}(K)$.
It follows from the above considerations that in $\mathcal{N}_{23}(K)$ a pair of matrices
$\widetilde{P}=(\widetilde{A}, \widetilde{B})$ is polynomially similar to a pair of matrices $P=(A,B)\ne 0$ if and only if it is similar to a pair of matrices
$P_{f,g}:=(f(A,B), g(A,B))$  for some  polynomials
\begin{equation}\label{polynomial}
\left.
\begin{array}{l}
f(x,y)= \alpha x +  \alpha_1y^2 + \alpha_2 xy, \\[2mm]
g(x,y)= \gamma x +\beta y+  \beta_1y^2 + \beta_2 xy,
\end{array}
\right.
\end{equation}
where, according to the condition (\ref{for not equal 0}), $\alpha, \beta\ne 0.$ Note that the constant terms are equal to zero according to the condition (\ref{for 00}), and the coefficient at $y$ in $f (x,y)$ is equal to zero according to $[f(A,B)]^2=0$; the equalities
$[g(A,B)]^3=0$ and $f(A,B)[g(A,B)]^2=0$ are always met and do not give any restrictions on the coefficients.

Note that in the  definition of  wildness
it is sufficient to require that the pairs $(M_1,N_1)$ and $(M_2,N_2)$ run not through the set
$\mathcal{M}(K)$, but through
the subset $E_1(K)$ of all pairs of matrices, all eigenvalues of which are equal to the identity element of the field (see, for example \cite{Gantmacher}, Chapter VIII).
This follows from the fact that the subset $E_1(K)$ is wild:
the following  pair of matrices of  $\mathcal{M}({F}_2)$ can be taken in this case
as  w-defining:
$$P=\left(\begin{array}{ccc}
I &I & 0 \\
0 &I & I \\
0 &0 &  I\\
\end{array}\right),
\hspace{4mm}
Q=\left(\begin{array}{ccc}
I &X & 0 \\
0 &I & Y \\
0 &0 &  I\\
\end{array}\right)
$$
(this can be verified by simple calculations).

Further, we
consider in $\mathcal{N}({F}_2)$  a pair of matrices
$P_0(X, Y)=(A_0, B_0)$ of the form
$$
A_0:= \left(\begin{array}{ccc|c|ccc}
0&0&0&0&I&0&0 \\
0&0&0&0&0&I&0 \\
0&0&0&0&0&0&I \\ \hline
0&0&0&0&0&0&0 \\ \hline
0&0&0&0&0&0&0 \\
0&0&0&0&0&0&0 \\
0&0&0&0&0&0&0 \\
\end{array}
\right),
\hspace{3mm}
B_0=B_0(X,Y):=  \left(\begin{array}{ccc|c|ccc}
0&0&I&0&0&0&0 \\
0&0&0&0&I&0&0 \\
0&0&0&0&0&T&0 \\ \hline
0&0&0&0&0&W&0 \\ \hline
0&0&0&0&0&0&I \\
0&0&0&0&0&0&0 \\
0&0&0&0&0&0&0 \\
\end{array}
\right),
$$
where
$$
T=T(X,Y):=
\left(\begin{array}{cc}
0&X\\
I&Y \\
\end{array}
\right),
\hspace{3mm}
W:=
\left(\begin{array}{cc}
0&I\\
\end{array}
\right),
$$
(here $X=X\cdot I, Y=Y\cdot I   $ for appropriate identity blocks  $I$).
Horizontal and vertical lines in the matrices
$A_0, B_0$ (which divide them into large blocks) are performed for formal reasons and, in particular, for a better understanding  ideas of the proposed calculations.
 It is easy to see that the pair  $P_0(X, Y)=(A_0, B_0)$  belongs to
$\mathcal{N}_{23}({F}_2)$  (for any  matrices $T$ and $W$).
We prove that the pair
$P_0(X,Y)$ is a  $w$-defining pair of matrices  of
$\mathcal{N}_{23}({F}_2)$ with  respect to $E_1(K)$
(see above the definition of wildness).

We denote
 $$A_{0f}:=f(A_0, B_0),   \ B_{0g}:=g(A_0,B_0)$$
 and  $P_0(X,Y)_{f,g}:=(f(A_0, B_0), g(A_0, B_0)).$

\begin{lemma}\label{lemma-1}
The  pair of matrices  $P_0(X,Y)_{f,g}\in\mathcal{N}_{23}({F}_2)$  is similar to a  pair of matrices
$\widehat{P}_0(X,Y)=(A_0, \widehat{B}_0)\in\mathcal{N}_{23}({F}_2)$, where  $\widehat{B}_0$  is of the
form
$$
\widehat{B}_0=\widehat{B}_0(X,Y):=
 \left(\begin{array}{ccc|c|ccc}
0 & 0 &I& 0 & *  & *           & *\\
0 & 0 & 0        & 0 & \dfrac{\beta^2}{\alpha}  I & *  & * \\
0 & 0 & 0        & 0 & 0        & \begin{array}{c}\dfrac{\beta}{\alpha} T\\[3mm] \end{array}  & *\\[0mm]
\hline
0 & 0 & 0        & 0 & 0        & W         & 0 \\
\hline
0 & 0 & 0        & 0 & 0        & 0                &  I \\
0 & 0 & 0        & 0 & 0        & 0                & 0 \\
0 & 0 & 0        & 0 & 0        & 0                & 0 \\
\end{array}
\right)
$$
with some elements $*$
(the exact form of which will not be used below$)$, $T=T(X,Y)$ and $W$ remain unchanged.
\end{lemma}

\begin{proof}
According to the definition of $A_{0f}$ and $B_{0g}$ and equalities
$$ f(x,y)= \alpha x +  \alpha_1y^2 + \alpha_2 xy, g(x,y)= \gamma x +\beta y+  \beta_1y^2 + \beta_2 xy  $$
 we have
$$A_{0f} = \left(\begin{array}{ccc|c|ccc}
0& 0   & 0  & 0 &   \alpha I&\alpha_1 T & \alpha_2 I  \\
0& 0   & 0  & 0 & 0&\alpha I&\alpha_1 I  \\
0& 0   & 0  & 0 &0&0&\alpha I\\
\hline
0&0&0&0& 0  & 0  & 0 \\
\hline
0&0&0&0 & 0 & 0  & 0 \\
0&0&0&0 &0  & 0  & 0 \\
0&0&0&0   &0& 0  & 0 \\
\end{array}\right),
\hspace{1mm}
B_{0g} =
  \left(\begin{array}{ccc|c|ccc}
0 & 0 & \beta I & 0 & \gamma I & \beta_1 T            & \beta_2 I \\
0 & 0 & 0        & 0 & \beta I & \gamma I & \beta_1 I \\
0 & 0 & 0        & 0 & 0        & \beta T         & \gamma I \\ \hline
0 & 0 & 0        & 0 & 0        & \beta W         & 0 \\ \hline
0 & 0 & 0        & 0 & 0        & 0                & \beta I \\
0 & 0 & 0        & 0 & 0        & 0                & 0 \\
0 & 0 & 0        & 0 & 0        & 0                & 0 \\
\end{array}
\right).
$$
Let us write $A_{0f}$ in the form
$$A_{0f} =  \left(\begin{array}{c|c|c}
0 & 0 & U \\
\hline
0 & 0 & 0 \\
\hline
0 & 0 & 0
\end{array}\right)$$
with  square  diagonal blocks    and
$$U =  \left(\begin{array}{ccc}
\alpha I   & \alpha_1 T          & \alpha_2 I \\
 0            & \alpha I      & \alpha_1 I \\
0             & 0                & \alpha E
\end{array}\right),
$$
and put
$$D= \left(\begin{array}{c|c|c}
U   & 0  & 0 \\
\hline
0   & I & 0 \\
\hline
0   & 0 &  I
\end{array}\right),$$
with such  a division into large  blocks, as in the matrices $A_0$ and $B_0$.
Since the inverse to the matrix $U$ is of the form
$$U^{-1} =  \left(\begin{array}{ccc}
\begin{array}{c}\dfrac{1}{\alpha} I\\[5mm]\end{array}
& \begin{array}{c}-\dfrac{\alpha_1}{\alpha^2} T\\[5mm]\end{array}
& \begin{array}{c}\dfrac{\alpha_1^2}{\alpha^3} T - \dfrac{\alpha_2}{\alpha^2} I\\[3mm]\end{array}  \\
\begin{array}{c} 0 \\[3mm]\end{array}                           & \begin{array}{c}\dfrac{1}{\alpha} I\\[3mm]\end{array}
 & \begin{array}{c} -\dfrac{\alpha_1}{\alpha^2} I\\[4mm]\end{array}\\
   0                            & 0                & \dfrac{1}{\alpha} I
\end{array}\right)$$

\vspace{-4mm}
we obviously have
$$D^{-1}A_{0f} D = A_0,$$

 $$D^{-1} B_{0g} D =
 \left(\begin{array}{ccc|c|ccc}
0 & 0 &\beta I& 0 & *  & *           & *\\
0 & 0 & 0        & 0 & \dfrac{\beta}{\alpha}  I & *  & * \\
0 & 0 & 0        & 0 & 0        & \dfrac{\beta}{\alpha} T  & * \\ \hline
0 & 0 & 0        & 0 & 0        &\beta W         & 0 \\ \hline
0 & 0 & 0        & 0 & 0        & 0                & \beta I \\
0 & 0 & 0        & 0 & 0        & 0                & 0 \\
0 & 0 & 0        & 0 & 0        & 0                & 0 \\
\end{array}
\right),
$$
(the elements denoted by $*$ are not used below,
 the block matrices $T=T(X,Y)$ and $W$ remain unchanged).

\vspace{3mm}

Now put
$$Z =  \left(\begin{array}{ccc|c|ccc}
 \beta I & 0 & 0 & 0 & 0 & 0 & 0 \\
0 & I & 0 & 0 & 0 & 0 & 0\\
0 & 0 & I & 0 & 0 & 0 & 0 \\  \hline
0 & 0 & 0 & \beta I & 0 & 0 & 0 \\  \hline
0 & 0 & 0 & 0 &  \beta I & 0 & 0 \\
0 & 0 & 0 & 0 & 0 & I & 0 \\
0 & 0 & 0 & 0 & 0 & 0 & I \\
\end{array}
\right)$$
and consider the matrices  $Z^{-1}A_{0f}Z$  and $Z^{-1}B_{0g}Z$.
The first of them remains  to be equal to $A_0$, and the second is equal to
  $\widehat{B}_0$,
specified in the condition of the lemma, which completes the proof.
\end{proof}

{ \textbf{The proof of  the main theorem.}}

As it was noted above we may require that the pairs  $(M_1,N_1)$ and $(M_2,N_2)$ run not through the set
$\mathcal{M}(K)$, but through its subset $E_1(K)$ of all pairs of matrices, whose  eigenvalues are equal to the identity element of the field (see above).  This class of pairs of matrices is wild as noted earlier.

  Let $(M_1,N_1), (M_1,N_2)\in E_1(K)$, and the pairs of matrices $P_0(M_1,N_1)$ and $P_0(M_2,N_2)$ be  polynomially similar
 (the matrices of the forms $T$ and $W$ are denoted for them by
 $T_1$, $W_1$ and $T_2$, $W_2$, respectively).
 Then by the definition of polynomial
similarity and  by Lemma \ref{lemma-1}
the pair $P_0(M_1,N_1)$, which consists of the matrices $A_0$ and $B_0=B_0(M_1,N_1)$, is similar to the pair $\widehat{P}(M_2,N_2) $ consisting of the matrices $A_0$ and $\widehat{B}=\widehat{B}(M_2,N_2)$.
Thus, there exists an invertible matrix $S$ such that the following equalities are fulfilled  :
\begin{equation}\label{eq-1}
A_0S=SA_0,
\end{equation}
\begin{equation}\label{eq-2}
\widehat{B_0}S=SB_0.
\end{equation}

  Denote by $B'_0$ and $\widehat{B}'_0$ the submatrices the of matrices
 $B_0$ and $\widehat{B}_0$ respectively, formed by the 1st, 2nd and 3rd horizontal and vertical stripes. Then these matrices and the matrix $A_0$
  are, up to permutations of rows and columns, direct sums of  $1\times 1$ and $2\times 2$ Jordan    cells with  zero eigenvalues. There are  well-known formulas   for finding solutions of the matrix equation $AX=XA$ with $A$ chosen in a Jordan normal form, see e.g.\cite{Gantmacher} (here $X$ is an arbitrary  matrix).  Using these formulas (or verifying directly) one can show that:

 $(a)$
equality (\ref{eq-1}) is equivalent to the equality
$$S =  \left(\begin{array}{c|c|c}
S_{11} &S_{12} & S_{13} \\
\hline
0         &S_{22} & S_{23} \\
\hline
0         & 0        & S_{11} \\
\end{array}\right)$$
for some matrices $S_{ij}$;

$(b)$ the equality $\widehat{B}'_0 S_{11}=S_{11}B'_0$
(relative to $S_{11}$ from $(a)$),
which follows from the  equality (\ref{eq-2}),
is equivalent to the equality
$$S_{11} =  \left(\begin{array}{ccc}
Y_{11} &Y_{12} & Y_{13} \\
0 &Y_{22} & Y_{23} \\
0 &0 &  Y_{11} \\
\end{array}\right)$$
for some matrices $Y_{ij}$.

Thus, the matrix $S$ is of the form

$$S =  \left(\begin{array}{ccc|c|ccc}
Y_{11} &Y_{12} & Y_{13} & Y_{14} &Y_{15} & Y_{16} & Y_{17} \\
0 &Y_{22} & Y_{23} & Y_{24} &Y_{25} & Y_{26} & Y_{27} \\
0 &0 &  Y_{11} & Y_{34} &Y_{35} & Y_{36} & Y_{37} \\ \hline
0 &0 & 0 & Y_{44} &Y_{45} & Y_{46} & Y_{47} \\ \hline
0 &0 &0 &0 &Y_{11} &Y_{12} & Y_{13} \\
0 &0 &0 &0 &0 &Y_{22} & Y_{23} \\
0 &0 &0 &0 &0 &0 & Y_{11}
\end{array}\right).$$

We continue to use (\ref{eq-2}) as the equality of block matrices.
By equating the matrix blocks standing in its left and right parts
at places (1,4), (2,5), (3,6) and (4,6), we obtain, respectively,
  the equalities

  $$Y_{34}=0, \
  \dfrac{\beta^2}{\alpha} Y_{11}= Y_{22}, \
  \dfrac{\beta}{\alpha} T_2Y_{22}=Y_{11}T_1+Y_{34}W_1$$
   and
  $W_2Y_{22}=Y_{44}W_1$,
 whence

 \begin{equation}\label{fe-1}
W_2Y_{22}=Y_{44}W_1,
\end{equation}
 \begin{equation}\label{fe-2}
\beta^3 T_2Y_{22}=\alpha^2Y_{22}T_1.
\end{equation}

It follows from the equality (\ref{fe-1}) that in the $2\times 2$ block matrix
$Y_{22}:=(Z_{ij})$ block $Z_{21}$ is zero.
Then, in expanded form, the equation (\ref{fe-2}) has the form
$$
\beta^3\left(\begin{array}{cc}
0&M_2\\
E&N_2 \\
\end{array}
\right)
\left(\begin{array}{cc}
Z_{11}&Z_{12}\\
0&Z_{22}\\
\end{array}
\right)=
\alpha^2\left(\begin{array}{cc}
Z_{11}&Z_{12}\\
0&Z_{22}\\
\end{array}
\right)
\left(\begin{array}{cc}
0&M_1\\
E&N_1 \\
\end{array}
\right)
$$
or, after matrix multiplication,
$$
\beta^3\left(\begin{array}{cc}
0&M_2Z_{22}\\
Z_{11}&Z_{12}+N_2Z_{22}\\
\end{array}
\right)=
\alpha^2\left(\begin{array}{cccc}
 Z_{12}&Z_{11}M_1+Z_{12}N_1\\
Z_{22}&Z_{22}N_1\\
\end{array}
\right).
$$
So we have the equalities
$$ Z_{12}=0,
\beta^3Z_{11}=\alpha^2Z_{22},$$
$$\beta^3M_2Z_{22}=\alpha^2Z_{11}M_1,
\beta^3N_2Z_{22}=\alpha^2Z_{22}N_1,$$ whence
$M_1=\dfrac{\beta^6}{\alpha^4}Z^{-1}_{22}M_2Z_{22}$ and
$N_1=\dfrac{\beta^3}{\alpha^2}Z^{-1}_{22}N_2Z_{22}$.

Since all the eigenvalues of the matrices $N_1$ and $N_2$
are equal to the identity element of the field $K$ we get
$\frac{\beta^3}{\alpha^2}=1$ and therefore
$$M_1=Z^{-1}_{22}M_2Z_{22},
N_1=Z^{-1}_{22}N_2Z_{22}.$$
Thus, the  pairs of matrices $(M_1,N_1)$ and $(M_2,N_2)$ are similar. The proof is complete.


%
\end{document}